\documentclass[11pt]{article}

\usepackage{amssymb,amsmath}
\usepackage{pstricks, pst-coil, pst-node, pst-tree, multido}
\usepackage[english]{babel}

\newtheorem{DE}{Definition}[section]

\newcommand {\sm} {\setminus}

\usepackage{latexsym} 
\usepackage{theorem} 
\usepackage{graphicx}

\newcommand{\qed}{\relax\ifmmode\hskip2em\Box\else\unskip\nobreak\hfill$\Box$\fi}

\newtheorem{theorem}[DE]{Theorem}
\newtheorem{lemma}[DE]{Lemma}

{\theoremstyle{break}\theorembodyfont{\rmfamily}}
{\theoremstyle{break}\theorembodyfont{\rmfamily}}

\newcounter{claim}
\newenvironment{proof}[1][]%
	{\noindent {\setcounter{claim}{0}\sc proof --- }{#1}{}}{\qed\vspace{2ex}}
	{\refstepcounter{claim}\vspace{1ex}\noindent {(\it\arabic{claim}) {#1}{}}\it}{\vspace{1ex}}
	{\noindent {}{#1}{}}{ This proves~(\arabic{claim}).\vspace{1ex}}

\begin{document}

\title{The (theta, wheel)-free graphs\\ Part I: only-prism and
  only-pyramid graphs}

\author{Emilie Diot\thanks{ENS de Lyon, LIP. E-mail:
    emilie.diot.pro@gmail.com}~, Marko Radovanovi\'c\thanks{University
    of Belgrade, Faculty of Mathematics, Belgrade, Serbia. Partially
    supported by Serbian Ministry of Education, Science and
    Technological Development project 174033. E-mail:
    markor@matf.bg.ac.rs}~, Nicolas Trotignon\thanks{CNRS, LIP, ENS de
    Lyon. Partially supported by ANR project Stint under reference
    ANR-13-BS02-0007 and by the LABEX MILYON (ANR-10-LABX-0070) of
    Universit\'e de Lyon, within the program ‘‘Investissements
    d'Avenir’’ (ANR-11-IDEX-0007) operated by the French National
    Research Agency (ANR).  Also Universit\'e Lyon~1, universit\'e de
    Lyon. E-mail: nicolas.trotignon@ens-lyon.fr}~, Kristina Vu\v{s}kovi\'c\thanks{School of Computing, University of Leeds, and
    Faculty of Computer Science (RAF), Union University, Belgrade,
    Serbia.  Partially supported by EPSRC grant EP/K016423/1, and
    Serbian Ministry of Education and Science projects 174033 and
    III44006. E-mail: k.vuskovic@leeds.ac.uk}}

\maketitle

\begin{abstract}
  Truemper configurations are four types of graphs (namely thetas,
  wheels, prisms and pyramids) that play an important role in the
  proof of several decomposition theorems for hereditary graph
  classes.  In this paper, we prove two structure theorems: one for
  graphs with no thetas, wheels and prisms as induced subgraphs, and
  one for graphs with no thetas, wheels and pyramids as induced
  subgraphs.  A consequence is a polynomial time recognition
  algorithms for these two classes.
  In Part II of this series we generalize these results to graphs with no thetas and wheels as induced subgraphs,
  and in Parts III and IV, using the obtained structure, we solve several optimization problems for these graphs.

\end{abstract}

\section{Introduction}\label{sec:intro}

In this article, all graphs are finite and simple.

A \emph{prism} is a graph made of three node-disjoint chordless paths
$P_1 = a_1 \dots b_1$, $P_2 = a_2 \dots b_2$, $P_3 = a_3 \dots b_3$ of
length at least 1, such that $a_1a_2a_3$ and $b_1b_2b_3$ are triangles
and no edges exist between the paths except those of the two
triangles.  Such a prism is also referred to as a
$3PC(a_1a_2a_3,b_1b_2b_3)$ or a $3PC(\Delta ,\Delta )$ (3PC stands for
\emph{3-path-configuration}).

A \emph{pyramid} is a graph made of three chordless paths
$P_1 = a \dots b_1$, $P_2 = a \dots b_2$, $P_3 = a \dots b_3$ of
length at least~1, two of which have length at least 2, node-disjoint
except at $a$, and such that $b_1b_2b_3$ is a triangle and no edges
exist between the paths except those of the triangle and the three
edges incident to $a$.  Such a pyramid is also referred to as a
$3PC(b_1b_2b_3,a)$ or a $3PC(\Delta ,\cdot)$.

A \emph{theta} is a graph made of three internally node-disjoint
chordless paths $P_1 = a \dots b$, $P_2 = a \dots b$,
$P_3 = a \dots b$ of length at least~2 and such that no edges exist
between the paths except the three edges incident to $a$ and the three
edges incident to $b$.  Such a theta is also referred to as a
$3PC(a, b)$ or a $3PC(\cdot ,\cdot)$.

\begin{figure}
  \begin{center}
    \includegraphics[height=2cm]{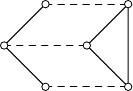}
    \hspace{.2em}
    \includegraphics[height=2cm]{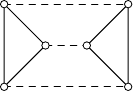}
    \hspace{.2em}
    \includegraphics[height=2cm]{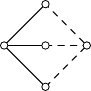}
    \hspace{.2em}
    \includegraphics[height=2cm]{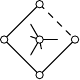}
  \end{center}
  \caption{Pyramid, prism, theta and wheel (dashed lines represent
    paths)\label{f:tc}}
\end{figure}

A \emph{hole} in a graph is a chordless cycle of length at least~4.
Observe that the lengths of the paths in the three definitions above
are designed so that the union of any two of the paths induce a hole.  A
\emph{wheel} $W= (H, c)$ is a graph formed by a hole $H$ (called the
\emph{rim}) together with a node $c$ (called the \emph{center}) that
has at least three neighbors in the hole.

A \emph{3-path-configuration} is a a graph isomorphic to a prism, a
pyramid or a theta.  A~\emph{Truemper configuration} is a graph
isomorphic to a prism, a pyramid, a theta or a wheel.  They appear in
a theorem of Truemper~\cite{truemper} that characterises graphs whose
edges can be labeled so that all chordless cycles have prescribed
parities (3-path-configurations seem to have first appeared in a
paper Watkins and Mesner~\cite{watkinsMesner:cycle}).

If $G$ and $H$ are graphs, we say that $G$ \emph{contains} $H$ when
$H$ is isomorphic to an induced subgraph of $G$.  We say that $G$ is
\emph{$H$-free} if it does not contain $H$.  We extend this to classes
of graphs with the obvious meaning (for instance, a graph is (theta,
wheel)-free if it does not contain a theta and does not contain a
wheel).

\medskip

Truemper configurations play an important role in the
analysis of several important hereditary graph classes,
as explained in a survey of Vu\v
skovi\'c~\cite{vuskovic:truemper}.  Let us simply mention here that
many decomposition theorems for classes of graphs are proved by
studying how some Truemper configuration contained in the graph
attaches to the rest of the graph, and often, the study relies on the
fact that some other Truemper configurations are excluded from the
class.  The most famous example is perhaps the class of \emph{perfect
  graphs}.  In these graphs, pyramids are excluded, and how a prism
contained in a perfect graphs attaches to the rest of the graph is
important in the decomposition theorem for perfect graphs, whose
corollary is the celebrated \emph{Strong Perfect Graph Theorem} due to
Chudnovksy, Robertson, Seymour and
Thomas~\cite{chudnovsky.r.s.t:spgt}.  See also~\cite{nicolas:perfect}
for a survey on perfect graphs, where a section is specifically
devoted to Truemper configurations.  But many other examples exist,
such as the seminal class of chordal graphs~\cite{dirac:chordal}
(containing no holes and therefore no Truemper configurations), universally signable
graphs~\cite{confortiCKV97} (which is exactly the class of graphs containing no Truemper
configurations), even-hole-free
graphs~\cite{conforti.c.k.v:eh1,kmv:evenhole} (containing pyramids but
not containing thetas and prisms), cap-free
graphs~\cite{conforti.c.k.v:capfree} (not containing prisms and
pyramids, but containing thetas), ISK4-free graphs~\cite{nicolas:isk4}
(containing prisms and thetas but not containing pyramids), chordless
graphs~\cite{mft:chordless} (containing no prisms, pyramids and
wheels, but containing thetas), (theta, triangle)-free
graphs~\cite{radovanovicV:theta} (containing no prisms, pyramids and
thetas), claw-free graphs~\cite{DBLP:conf/bcc/ChudnovskyS05}
(containing prisms, but not containing pyramids and thetas) and
bull-free graphs~\cite{chudnovsky12} (containing thetas and the prism
on six nodes, but not containing pyramids and prisms on at least 7
nodes).  In most of these classes, some wheels are allowed and some
are not.  In some of them (notably perfect graphs and even-hole-free
graphs), the structure of a graph containing a wheel is an important
step in the study of the class.  Let us mention that the classical
algorithm LexBFS produces an interesting ordering of the nodes in many
classes of graphs where some well-chosen Truemper configurations are
excluded~\cite{abChTrVu:moplex}.  Let us also mention that many
subclasses of wheel-free graphs are well studied, namely unichord-free
graphs~\cite{nicolas.kristina:one}, graphs that do not contain $K_4$
or a subdivision of a wheel as an induced
subgraph~\cite{nicolas:isk4}, graphs that do not contain $K_4$ or a
wheel as a subgraph~\cite{thomassenToft:k4,aboulkerHT:wheelFree},
propeller-free graphs~\cite{aboulkerRTV:propeller}, graphs with no
wheel or antiwheel~\cite{maffray:15} and planar wheel-free
graphs~\cite{aboulker.c.s.T:wfpg}.

\medskip

All these examples suggest that a systematic study of classes of graphs defined
by excluding Truemper configurations  is of interest. It
might shed a new light on all the classes mentioned above and be interesting in
its own right.  In this paper we study two of such classes.
Since there are four types of Truemper configurations, there
are potentially $2^4= 16$ classes of graphs defined by excluding them
(such as prism-free graphs, (theta, wheel)-free graphs, and so on).
In one of them, none of the Truemper configurations are excluded, so
it is the class of all graphs.  We are left with 15 non-trivial
classes where at least one type of Truemper configuration is excluded.
One case is when all Truemper configurations are excluded.  This class
is known as the class of \emph{universally signable
  graphs}~\cite{confortiCKV97} and it is well studied: its structure
is fully described, and many difficult problems such as graph
coloring, and the maximum clique and stable set problems can be solved
in polynomial time for this class (see~\cite{abChTrVu:moplex} for the
most recent algorithms for them).  So we are left with 14 classes of
graphs, and to the best of our knowledge, they were not studied so
far, except for one aspect: the complexity of the recognition problem
is known for 11 of them.  Let us survey this.

\newcommand{\zero}{excluded} \newcommand{\un}{---}

It is convenient to sum up in a table all the 16 classes.  In
Table~\ref{t:t}, each line of the table represents a class of graphs
defined by excluding some Truemper configurations.  The first four
columns indicate which Truemper configurations are excluded and which
are allowed.  The last columns indicates the complexity of the
recognition algorithm and a reference to the paper where this
complexity is proved.  Lines with a reference to a theorem indicate a
result proved here. For instance line~5 of the table should be read as
follows: the complexity of deciding whether a graph is in the class of
(theta, prism)-free graphs is $O(n^{35})$ (throughout the paper, $n$
stands for the number of nodes, and $m$ for the number of edges of the
input graph).  Observe that a recognition algorithm for (theta,
prism)-free graphs is equivalent to an algorithm to decide whether a
graph contains a theta \emph{or} a prism.  Note that all the proofs of
NP-completeness rely on a variant of a classical construction of
Bienstock~\cite{bienstock:evenpair}.

\begin{table}
  \begin{tabular}{ccccccc}
    k & theta & pyramid & prism & wheel & Complexity & Reference\\ \hline
    0   &  \zero  & \zero & \zero & \zero & $O(nm)$&\cite{confortiCKV97}\cite{tarjan:clique}\\
    1   &  \zero  & \zero & \zero & \un & $O(n^7)$& \cite{maffray.t:reco}\cite{maffray.t.v:3pcsquare} \\
    2   &  \zero  & \zero & \un & \zero & $O(n^3m)$ & Theorem~\ref{th:recoOprism}\\
    3   &  \zero  & \zero & \un & \un & $O(n^7)$&\cite{maffray.t.v:3pcsquare}\\
    4   &  \zero  & \un & \zero & \zero & $O(n^4m)$& Theorem~\ref{th:recoOpy}\\
    5   &  \zero  & \un & \zero & \un & $O(n^{35})$&\cite{Chudnovsky.Ka:08} \\
    6   &  \zero  & \un & \un & \zero & $O(n^4m)$&Part II \cite{rtv} \\
    7   &  \zero  & \un & \un & \un & $O(n^{11})$&\cite{chudnovsky.seymour:theta}\\
    8   &  \un  & \zero & \zero & \zero & NPC&\cite{diotTaTr:13}\\
    9   &  \un  & \zero & \zero & \un & $O(n^5)$&\cite{maffray.t:reco} \\
    10   &  \un  & \zero & \un & \zero & NPC& \cite{diotTaTr:13}\\
    11   &  \un  & \zero & \un & \un & $O(n^9)$&\cite{chudnovsky.c.l.s.v:reco} \\
    12   &  \un  & \un & \zero & \zero & NPC&\cite{diotTaTr:13}\\
    13   &  \un  & \un & \zero & \un & NPC&\cite{maffray.t:reco}\\
    14   &  \un  & \un & \un & \zero & NPC&\cite{diotTaTr:13}\\
    15   &  \un  & \un & \un & \un & $O(1)$&Trivial\\
  \end{tabular}
  \caption{Detecting Truemper configurations\label{t:t}}
\end{table}

As already stated, 13 of the recognition problems of Table~\ref{t:t}
are solved in previous work.
In this paper and its subsequent part \cite{rtv} we resolve the complexity
of recognition of the remaining three classes.
In this paper we give a polynomial time recognition
algorithm for the following two classes: (theta, wheel, pyramid)-free
and (theta, wheel, prism)-free graphs.  In the first class, the only
allowed Truemper configurations are prisms, and in the second, the only
ones are pyramids.  We therefore use the names \emph{only-prism} and
\emph{only-pyramid} for these two classes.  The last problem from
Table~\ref{t:t}, namely the recognition of (theta, wheel)-free
graphs, a similar approach is successful while being more
complicated. This class is studied in a subsequent paper by the last
three authors \cite{rtv}.

For each class, our recognition algorithm relies on a decomposition
theorem for the class.  In each case, this theorem fully describes the
structure of the most general graph in the class, and could therefore
be used to provide algorithms for several combinatorial optimisation
problems.  This is done in Parts III and IV of this series (see \cite{rtv3} and \cite{rtv4}), where  polynomial-time algorithms for finding maximum weighted clique and stable set, for optimal coloring and for induced version of $k$-linkage problem (for $k$ fixed) are obtained for the class of (theta,wheel)-free graphs.  We note that among the 16 classes described in
Table~\ref{t:t}, only universally signable graphs (line~0 from the
table) have a (previously known) decomposition theorem.  All the other (previously known) polynomial time
algorithms mentioned in Table~\ref{t:t} are based on a direct
algorithm to detect the obstruction.

In Section~\ref{sec:mr}, we give some notation and we describe the
results, in particular we state precisely the decomposition theorems
proved in the rest of the paper. In Section~\ref{sec:pl}, we prove
several lemmas needed in many places. In Section~\ref{sec:onlyprism},
we prove the decomposition theorem for only-prism
graphs. In Section~\ref{sec:onlypyramid}, we prove the decomposition
theorem for only-pyramid  graphs (note that the proof
relies mostly on theorems proved previously in~\cite{kmv:evenhole}).
In Section~\ref{sec:2joins}, we prove that the 2-joins (a
decomposition defined in the next section) that actually occur in our
classes of graph have a special structure. In Section~\ref{sec:algo},
we describe the recognition algorithms and show how the decomposition theorems
that we prove can be transformed into structure theorems.

\section{Main results}
\label{sec:mr}

A {\em path} $P$ is a sequence of distinct nodes $p_1p_2\ldots p_k$,
$k\geq 1$, such that $p_ip_{i+1}$ is an edge for all $1\leq i <k$.
Edges $p_ip_{i+1}$, for $1\leq i <k$, are called the {\em edges of
  $P$}.  Nodes $p_1$ and $p_k$ are the {\em ends} of $P$.  A cycle $C$
is a sequence of nodes $p_1p_2\ldots p_kp_1$, $k \geq 3$, such that
$p_1\ldots p_k$ is a path and $p_1p_k$ is an edge.  Edges
$p_ip_{i+1}$, for $1\leq i <k$, and edge $p_1p_k$ are called the {\em
  edges of $C$}.  Let $Q$ be a path or a cycle.  The node set of $Q$
is denoted by $V(Q)$.  The {\em length} of $Q$ is the number of its
edges.  An edge $e=uv$ is a {\em chord} of $Q$ if $u,v\in V(Q)$, but
$uv$ is not an edge of $Q$. A path or a cycle $Q$ in a graph $G$ is
{\em chordless} if no edge of $G$ is a chord of $Q$.  For a path $P$
and $u,v\in V(P)$, we denote with $uPv$ the chordless path in $P$ from
$u$ to $v$.

A subset $S$ of nodes of a graph $G$ is a {\em cutset} if $G\sm S$ is
disconnected.  A {\em clique} in a graph is a (possibly empty) set of pairwise adjacent vertices.
A clique on $k$ nodes is denoted by $K_k$.  A $K_3$ is
also referred to as a {\em triangle}, and is denoted by $\Delta$.  A
node cutset $S$ is a {\em clique cutset} if $S$ is a clique.
Note that in particular  the empty set is a clique and that a
disconnected graph has a clique cutset (the empty set).

Our main results are generalizations of the next two theorems. A graph
is \emph{chordal} if it is hole-free.

\begin{theorem}[Dirac \cite{dirac:chordal}]
  \label{th:chordal}
  A chordal graph is either a clique or has a clique cutset.
\end{theorem}

\begin{theorem}[Conforti, Cornu\'ejols, Kapoor, Vu\v skovi\'c \cite{confortiCKV97}]\label{universally}
  A (theta, wheel, pyramid, prism)-free graph is either a clique or a
  hole, or has a clique cutset.
\end{theorem}

To state the next theorem, we need the notion of a line graph.  If $R$
is a graph, then the \emph{line graph} of $R$ is the graph $G$ whose
nodes are the edges of $R$ and such that two nodes of $G$ are
adjacent in $G$ whenever they are adjacent edges of $R$.  We write
$G=L(R)$.

We need several results about line graphs.  A {\it diamond} is a
graph obtained from a $K_4$ by deleting an edge. A {\it claw} is a
graph induced by nodes $u,v_1,v_2,v_3$ and edges $uv_1,uv_2,uv_3$.

\begin{theorem}[Harary and Holzmann \cite{harary.holzmann:lgbip}]
  \label{th:HararyH}
  A graph is (claw, diamond)-free graph if and only if it is the line
  graph of a triangle-free graph.
\end{theorem}

The following characterises the line graphs that actually appear in
our classes.  A graph $G$ is \emph{chordless} if every cycle of $G$ is
chordless.  Note that chordless graphs have a full structural
description not needed here and explained
in~\cite{aboulkerRTV:propeller}.

\begin{lemma}
  \label{l:lgtfc}
  For a graph $G$, the following three conditions are equivalent.
  \begin{enumerate}
  \item\label{i:lg1} $G$ is a (wheel, diamond)-free line graph.
  \item\label{i:lg2} $G$ is the line graph of a triangle-free chordless graph.
  \item\label{i:lg3} $G$ is (wheel, diamond, claw)-free.
  \end{enumerate}
\end{lemma}

\begin{proof}
  \ref{i:lg1}$\rightarrow$\ref{i:lg2}.  Let $R$ be such that
  $G=L(R)$.  For every connected component of $R$ that is isomorphic
  to a triangle, we erase the triangle and replace it by a claw.  This
  yields a graph $R'$ and $G= L(R') = L(R)$ because a claw and a triangle
  have the same line graph.  We claim that $R'$ is triangle-free, so
  suppose for a contradiction that $R'$ contains a triangle $T = abc$.  By
  the construction of $R'$, $T$ is not a connected component of $R'$,
  so there exists a node $d$ not in $T$ with a neighbor in $T$, say
  $a$.  Now the edges $ab, bc, ac, da$ of $R'$ induce a diamond in
  $G$, a contradiction.  Also $R'$ is chordless because the edge set
  of a cycle together with a chord of that cycle in $R'$ yields a
  wheel in $L(R')$ (centred at the chord).

  \ref{i:lg2}$\rightarrow$\ref{i:lg3}. Since $G$ is the line graph of
  a triangle-free graph $R$, by Theorem~\ref{th:HararyH}, $G$ is (diamond,
  claw)-free.  Suppose for a contradiction that $G$ contains a wheel
  $(H, c)$.  Let $H= v_1\dots v_k v_1$.  So, in $R$ and with
  subscripts taken modulo $k$,   $v_1, \dots, v_k$ are
  edges of $R$, and for $i=1, \dots, k$, $v_i$ is adjacent to
  $v_{i+1}$ and $v_{i-1}$, and to no other edges among the $v_j$'s
  since $H$ is a hole.  It follows that $v_1, \dots, v_k$ are the
  edges of a cycle $C$ of $R$.  Now, $c$ is an edge of $R$ that is
  adjacent to at least three edges of $C$.  It is therefore a chord of
  $C$, a contradiction.

  \ref{i:lg3}$\rightarrow$\ref{i:lg1}. Since $G$ is (diamond,
  claw)-free, it is a line graph by Theorem~\ref{th:HararyH}, and it
  is (wheel, diamond)-free by assumption.
\end{proof}

Our first decomposition theorem is the following. The proof is given
in Section~\ref{sec:onlyprism}.  Note that by Lemma~\ref{l:lgtfc}, the
line graph of a triangle-free chordless graph is only-prism
(because every pyramid and every theta contains a claw).

\begin{theorem}
  \label{th:decWTPFRee}
  If $G$ is an only-prism graph, then $G$ is the line graph of a
  triangle-free chordless graph or $G$ admits a clique cutset.
\end{theorem}

To state the next theorem, we need a new basic class and a new
decomposition that we define now.  We start with the basic class.

An edge of a graph is \emph{pendant} if one of its ends has degree~1.
Two pendant edges of a  tree $T$ are \emph{siblings} if the unique path
of $T$ linking them contains at most one node of degree at least~3.  A
tree is \emph{safe} if for every node $u$ of degree 1, the neighbor
$v$ of $u$ has degree at most 2 and $uv$ has at most one sibling.  A
\emph{pyramid-basic} graph is any graph
$G$ constructed as follows:

\begin{itemize}
\item Consider a safe tree $T$ and give to each pendant edge of $T$
  a label $x$ or $y$, in such a way that for every pair of siblings,
  distinct labels are given to the members of the pair.
\item Build the line graph $L(T)$, and note that since the nodes of
  $L(T)$ are the edges of $T$, some nodes of $L(T)$ have a label (they
  are the nodes of degree~1 of $L(T)$).
\item Construct $G$ from $L(T)$ by adding a node $x$ adjacent to
  every node with label $x$, and a node $y$ adjacent to $x$ and to
  every node with label~$y$.
\end{itemize}

\begin{lemma}
  \label{l:opinC}
  Every pyramid-basic graph is only-pyramid.
\end{lemma}

\begin{proof}
  Let $G$ be constructed as above.

  Since $L(T)$ is claw-free by Theorem~\ref{th:HararyH}, and every node of $L(T)$ with a label
  has degree~1 in $L(T)$ and degree~2 in $G$, we see that no node in
  $G$ apart from $x$ and $y$ can be the center of a claw.  It follows
  that the centers of claws in $G$ form a clique, so $G$ cannot
  contain a theta.

  Suppose for a contradiction that $G$ contains a prism, say a
  3PC($a_1a_2a_3$, $b_1b_2b_3$).  Note that $x$ and $y$ are not
  contained in any triangle of $G$, so $a_1, a_2, a_3, b_1, b_2, b_3$
  are all members of $L(T)$.  In $T$, $a_1, a_2, a_3$ are edges with a
  common end $a$ and $b_1, b_2, b_3$ are edges with a common end $b$.
  In $T$, there is a cut-edge $e$ separating $a$ and $b$.  So,
  $\{e, x, y\}$ is a node-cut of $L(T)$ that separates
  $\{a_1, a_2, a_3\}\sm \{e\}$ from $\{b_1, b_2, b_3\} \sm \{e\}$. It
  follows that one path of the prism goes through $x$ while another
  path goes through $y$.  This is a contradiction since $x$ and $y$
  are adjacent.

  To prove that $G$ is wheel-free, we study the holes of $G$.  Let $H$
  be a hole of $G$.  Since $L(T)$ contains no hole, $H$ must contain
  $x$ or $y$.  If it contains exactly one of them, say $x$ up to
  symmetry, then $H= x p_1\dots p_kx$ and $p_1$ and $p_k$ have
  degree~1 in $L(T)$.  Since all neighbors of $y$ in $L(T)$ have
  degree 1 in $L(T)$, $y$ has no neighbor in $H\sm x$.  A node of
  $L(T)$ not in $H$ is an edge of $T$ that can be adjacent (in $T$) to
  at most two edges among $p_1, \dots, p_k$ (that are indeed edges of
  $T$). And if it is adjacent to two edges, it is a non-pendant node
  in $L(T)$, so it is non-adjacent to $x$.  It follows that no node
  of $G$ can be the center of wheel with rim $H$.

  If $H$ goes through $x$ and $y$, then again $H= x y p_1\dots p_kx$
  and $p_1$ and $p_k$ have degree~1 in $L(T)$.  As above, no node of
  $G$ can have three neighbors in $H$. It follows that $G$ is
  wheel-free.
\end{proof}

We now define a decomposition that we need.
A graph $G$ has an {\em almost 2-join} $(X_1,X_2)$ if $V(G)$ can be
partitioned into sets $X_1$ and $X_2$ so that the following hold:
\begin{itemize}
\item For $i=1,2$, $X_i$ contains disjoint nonempty sets $A_i$ and
  $B_i$, such that every node of $A_1$ is adjacent to every node
  of $A_2$, every node of $B_1$ is adjacent to every node of
  $B_2$, and there are no other adjacencies between $X_1$ and $X_2$.
\item For $i=1,2$, $|X_i|\geq 3$.
\end{itemize}

An almost 2-join $(X_1, X_2)$ is a \emph{2-join} when for $i=1,2$, $X_i$
contains at least one path from $A_i$ to $B_i$, and if $|A_i|=|B_i|=1$
then $G[X_i]$ is not a chordless path.

We say that $(X_1,X_2,A_1,A_2,B_1,B_2)$ is a {\em split} of this
2-join, and the sets $A_1,A_2,B_1,B_2$ are the {\em special sets} of
this 2-join.  We often use the following notation: $C_i = X_i\sm (A_i
\cup B_i)$ (possibly, $C_i = \emptyset$).

A pyramid is {\em long} if all of its paths are of length at least~2
(note that the long pyramids are precisely the wheel-free pyramid).
Our second decomposition theorem is the following.  It is proved in
Section~\ref{sec:onlypyramid}.

\begin{theorem}
  \label{ptwfree}
An only-pyramid graph is either one of the following graphs:
\begin{itemize}
\item a clique,
\item a hole,
\item a long pyramid, or
\item a pyramid-basic graph,
\end{itemize}
or it has a clique cutset or a 2-join.
\end{theorem}

In Section~\ref{sec:algo}, we will show that the two theorems above
in fact lead to structure theorems: they can be turned into a method that
actually allows us to build every graph in the class that they describe.

\section{Preliminary lemmas}\label{sec:pl}
\label{sec:pl}

\begin{lemma}
  \label{l:duk}
  If $G$ is a diamond-free graph then every edge of $G$ is
  contained in a unique maximal clique of $G$.
\end{lemma}

\begin{proof}
  An edge $uv$ is obviously in at least one maximal clique.  If it is
  not unique, then let $K$ and $K'$ be two distinct maximal cliques
  containing $uv$.  Since by maximality $K \not\subseteq K'$,  there
  exists $w \in K \sm K'$.  By the maximality  of $K'$, there exists
  in $K'$ a non-neighbor $w'$ of $w$. So, $\{u, v, w, w'\}$ induces a
  diamond, a contradiction.
\end{proof}

When $C$ and $H$ are two disjoint sets of nodes of a graph (or
induced subgraphs), we say that $C$ is {\em $H$-complete}, if every
node of $C$ is adjacent to every node of $H$.

\begin{lemma}\label{diamond}
  If $G$ is a wheel-free graph that contains a diamond, then $G$ has a
  clique cutset.
\end{lemma}

\begin{proof}
  Let $K$ be a clique of size at least~2 in $G$, such that there exist
  two nodes in $G\setminus K$, non-adjacent and $K$-complete.
  Observe that $K$ exists because $G$ contains a diamond.  Suppose
  that $K$ is maximal with respect to this property.  We now prove
  that $K$ is a clique cutset of $G$.  Otherwise, for every pair
  $a, b \in V(G) \setminus K$ of non-adjacent $K$-complete nodes
  there exists a path $P$ from $a$ to $b$ in $G\setminus K$.  Let
  $(a, b, P)$ be a triple as above and chosen subject to the
  minimality of $P$.  If no internal node of $P$ has a neighbor in
  $K$, then for any pair $x, y \in K$, $V(P) \cup \{x, y\}$ induces a
  wheel, a contradiction.  So, let $c$ be the internal node of $P$
  closest to $a$ along $P$ that has a neighbor $x$ in $K$.  We claim
  that $c$ has a non-neighbor $y$ in $K$.  Otherwise, one of the
  triple $(a, c, aPc)$ or $(c, b, cPb)$ contradicts the minimality of
  $P$, unless $P$ has length~2.  In this case, $K\cup \{c\}$
  contradicts the maximality of $K$.  So, our claim is proved.  Let
  $d$ be the neighbor of $y$ in $P$ closest to $c$ along $P$ (note
  that $c\neq d$, so $y d P a y$ has length at least~4).  Now,
  $(y d P a y, x)$ is a wheel, a contradiction.
\end{proof}

A \emph{star cutset} in a graph is a node-cutset $S$ that contains a node
(called a \emph{center}) adjacent to all other nodes of $S$.  Note
that a nonempty clique cutset is a star cutset.

\begin{lemma}\label{star}
  If a (theta, wheel)-free graph $G$ has a star cutset, then
  $G$ has a clique cutset.
\end{lemma}
\begin{proof}
  Let $S$ be a star cutset centred at $x$, and assume that it is a
  minimal such cutset, i.e.  no proper subset of $S$ is a star cutset
  of $G$ centred at $x$. We now show that $S$ induces a clique. Assume
  not, and let $u$ and $v$ be two nonadjacent nodes of $S$.  Let $C_1$
  and $C_2$ be two of the connected components of $G\setminus S$.  By
  the choice of $S$, both $u$ and $v$ have neighbors in both $C_1$ and
  $C_2$.  So for $i=1,2$, there is a chordless $uv$-path $P_i$ in
  $G[C_i\cup \{ u,v \}]$. But then $P_1\cup P_2\cup x$ induces a theta
  or a wheel with center $x$.
\end{proof}

\begin{lemma}\label{l:Hv}
  Let $G$ be a (theta, wheel)-free graph. If $H$ is a hole of $G$ and
  $v$ a node of $V(G)\setminus V(H)$,  then $v$ has at most two
  neighbors in $H$, and if it has two neighbors in $H$, then they are
  adjacent.
\end{lemma}
\begin{proof}
  Node $v$ has at most two neighbors in $H$, since otherwise $(H,v)$
  is a wheel.  If $v$ has two nonadjacent neighbors in $H$, then
  $H\cup \{ v\}$ induces a theta.
\end{proof}

\section{Only-prism graphs}
\label{sec:onlyprism}

In this section we prove Theorem~\ref{th:decWTPFRee}.

\begin{lemma}
  \label{l:2holesT}
  Let $G$ be an only-prism graph.  Suppose that $G$ contains
  two chordless paths $P = x_P \dots y_P$ and $Q = x_Q \dots z_Q$, of length at
  least 1, node disjoints, with no edges between them.  Suppose $x,
  y, z \notin V(P)\cup V(Q)$ are pairwise adjacent and such that $N(x)
  \cap (V(P) \cup V(Q)) = \{x_P, x_Q\}$, $N(y) \cap (V(P) \cup V(Q)) = \{y_P\}$
  and $N(z) \cap (V(P) \cup V(Q)) = \{z_Q\}$.  Then, $G$ has a clique cutset.
\end{lemma}

\begin{proof}
  By Lemma~\ref{diamond}, we may assume that $G$ is diamond-free.  So,
  by Lemma~\ref{l:duk}, there exists a unique maximal clique $K$ of
  $G$ that contains $x$, $y$ and $z$.  Suppose that $K$ is not a
  clique cutset.  So, $G \setminus K$ contains a shortest path
  $R = u \dots v$ such that $u$ has a neighbor in $P$, and $v$ has a
  neighbor in $Q$.  From the minimality of $R$, $R\setminus u$ has no
  neighbors in $P$ and $R\setminus v$ has no neighbors in $Q$.  We set
  $P_x = x x_P P y_P$, $P_y = y y_P P x_P$, $Q_x = x x_Q Q z_Q$,
  $Q_z = z z_Q Q x_Q$.  Let $u_x$ (resp.\ $u_y$) be the neighbor of
  $u$ in $P_x$ (resp.\ in $P_y$) closest to $x$ (resp.\ to $y$) along
  $P_x$ (resp.\ along $P_y$).  Let $v_x$ (resp.\ $v_z$) be the
  neighbor of $v$ in $Q_x$ (resp.\ in $Q_z$) closest to $x$ (resp.\ to
  $z$) along $Q_x$ (resp.\ along $Q_z$).  By Lemma~\ref{l:Hv} applied
  to $u$ and the hole $x x_P P y_P y x$, either $u_x = u_y$ and
  $u_x \notin \{x, y\}$, or $u_xu_y \in E(G)$ and
  $\{u_x, u_y\} \neq \{x, y\}$.  Similarly, either $v_x = v_z$ and
  $v_x \notin \{x, z\}$, or $v_xv_z \in E(G)$ and
  $\{v_x, v_z\} \neq \{x, z\}$.

  Note that every node of $R$ has at most one neighbor in $\{x, y,
  z\}$ because $G$ is diamond-free and $K$ is maximal.
  Suppose that $x$ has a
  neighbor $r\in R$. Let $Y$ be a shortest path from $r$ to $y$ in $(uRr)
  \cup P_y$, and $Z$ a shortest path from $r$ to $z$ in $(rRv) \cup Q_z$.
  Since $Y \cup Z \cup \{ x \}$ cannot induce a wheel with center $x$, w.l.o.g.
  $y$ has a neighbor in $Z\cap R$. Let $y'$ be such a neighbor closest to $r$.
  Note that $y'\neq r$. Let $H$ be the hole induced by $Y$ and
  $rRy'$. Then $x$ and $H$ contradict Lemma~\ref{l:Hv}.
 So $x$ has no neighbor in
  $R$,  and in particular $x \notin \{u_x, v_x\}$.

  Now let $H$ be the hole induced by $xP_xu_x$, $xQ_x v_x$ and $R$. If $y$ has a neighbor
  $r$ in $R$, then since $x$ is not adjacent to $r$, hole $H$ and node $y$
  contradict Lemma~\ref{l:Hv}. Therefore $y$ has no neighbors in $R$,
  and by symmetry neither does $z$.

  If $u_xu_y \in E(G)$, then the three paths $x Q_x v_x v R u$, $x P_x
  u_x$ and $x y P_y u_y$ form a pyramid, a contradiction.  Therefore,
  as noted above, $u_x = u_y$ and $u_x \notin \{x, y\}$.
  If $u_xx \in E(G)$ then
  $R$, $P_y$ and $v_z Q_z z$ form the rim of a wheel centered
  at $x$.  So, $u_xx \notin E(G)$.  It follows that
  the three paths $u_x P_x x$, $u_x P_y yx$ and $u_x u R v v_x Q_xx$
  form a theta, a contradiction.
\end{proof}

\begin{lemma}
  \label{l:2holes}
  Let $G$ be an only-prism graph.  Suppose that $G$
  contains two chordless paths $P = x_P \dots y_P$ and $Q = x_Q \dots y_Q$, of
  length at least 1, node disjoints, with no edges between them.
  Suppose $x, y \notin V(P)\cup V(Q)$ are adjacent and such that $N(x)
  \cap (V(P) \cup V(Q)) = \{x_P, x_Q\}$ and $N(y) \cap (V(P) \cup V(Q)) = \{y_P,
  y_Q\}$.  Then, $G$ has a clique cutset.
\end{lemma}

\begin{proof}
  By Lemma~\ref{diamond}, we may assume that $G$ is diamond-free.
  So,   by Lemma~\ref{l:duk} there exists a unique maximal clique $K$ of $G$ that contains
  $x$ and $y$.  Observe that all common neighbors of $x$ and $y$ are
  in $K$.  Suppose that $K$ is not a clique cutset.  So, $G \setminus
  (H\cup K)$ contains a shortest path $R = u \dots v$ such that $u$
  has a neighbor in $P$, and $v$ has a neighbor in $Q$.  We suppose that $P, Q,
  R$ are minimal w.r.t.\ all the properties above.

  From the minimality of $R$, $R\setminus u$ has no neighbors in~$P$
  and $R\setminus v$ has no neighbors in~$Q$.  We set
  $P_x = x x_P P y_P$, $P_y = y y_P P x_P$, $Q_x = x x_Q Q y_Q$,
  $Q_y = y y_Q Q x_Q$.  Let $u_x$ (resp.\ $u_y$) be the neighbor of
  $u$ in $P_x$ (resp.\ in $P_y$) closest to $x$ (resp.\ to $y$) along
  $P_x$ (resp.\ along $P_y$).  Let $v_x$ (resp.\ $v_y$) be the
  neighbor of $v$ in $Q_x$ (resp.\ in $Q_y$) closest to $x$ (resp.\ to
  $y$) along $Q_x$ (resp.\ along $Q_y$).  By Lemma~\ref{l:Hv} applied
  to $u$ and the hole $x x_P P y_P y x$, either $u_x = u_y$ and
  $u_x \notin \{x, y\}$, or $u_xu_y \in E(G)$ and
  $\{u_x, u_y\} \neq \{x, y\}$.  Similarly, either $v_x = v_y$ and
  $v_x \notin \{x, y\}$, or $v_xv_y \in E(G)$ and
  $\{v_x, v_y\} \neq \{x, y\}$.

  Suppose that both $x$ and $y$ have neighbors in the interior of $R$.
  So, there is a shortest path $R'$ in the interior of $R$ linking a
  neighbor $r$ of $x$ to a neighbor $r'$ of $y$.  Observe that $R'$
  has length at least 1, because every common neighbor of $x$ and $y$ is in
  $K$.  Hence, $P, R', uRr$ contradict the minimality of $P, Q, R$.
  So, we may assume up to symmetry that $y$ has no neighbor in the
  interior of $R$.  If $x$ has a neighbor in the interior of $R$, in
  particular $R$ has length at least~2, so $y P_y u_y u R v v_y Q_y y$
  is a hole, and $x$ has two non-adjacent neighbors in it (namely $y$
  and some internal node of $R$), a contradiction to
  Lemma~\ref{l:Hv}.  Hence, $x$ and $y$ have no neighbors in the
  interior of $R$.

  Suppose that $u_xu_y \in E(G)$.  If $u_x = x$, then $u_y\neq y$ and
  $(u_y u R v v_y Q_y y P_y u_y, x)$ is a wheel, a contradiction.  So,
  $u_x\neq x$. By symmetry it follows that $u$ (resp. $v$) is not
  adjacent to $x$ nor $y$.  Since $Q$ has length at least 1, it is
  impossible that $v_xy\in E(G)$ and $v_yx\in E$, so suppose up to
  symmetry that $v_yx \notin E(G)$.  The three paths
  $y Q_y v_y v R u$, $yx P_x u_x$ and $yP_y u_y$ form a pyramid, a
  contradiction.  Therefore, as noted above, $u_x = u_y$ and
  $u_x \notin \{x, y\}$.  Similarly, $v_x = v_y$ and
  $v_x \notin \{x, y\}$.

  We may assume w.l.o.g.\ that $u_xx\notin E(G)$. If
  $v_xy\not\in E(G)$ then the three paths $xP_xu_x$, $xyP_yu_x$ and
  $xQ_xv_xvRuu_x$ form a theta, a contradiction. So $v_xy\in E(G)$, and
  by symmetry it follows that $u_xy\in E(G)$. But then $P$, $Q$, $R$
  and $\{ x,y \}$ induce a wheel with center $y$, a contradiction.
\end{proof}

\begin{lemma}
  \label{l:HU}
    If $G$ is an only-prism graph, $H$ is a hole in
    $G$, and $x\in V(G) \setminus V(H)$ has a unique neighbor in $V(H)$,
    then $G$ has a clique cutset.
\end{lemma}

\begin{proof}
  Let $y$ be the unique neighbor of $x$ in $H$.  If $y$ is not a
  cutnode of $G$, then some path $P = u \dots v$ of $G \setminus (H
  \cup \{x\})$ is such that  $u$
  is adjacent to $x$, and $v$  has a neighbor in $H\setminus y$.  We
  suppose that $H, x, P$ are minimal subject to all the properties
  above.

  Suppose that some node $v'$ of $P$ is adjacent to $y$.  If
  $v'\neq v$, then by the minimality of $P$, $v'$ has a unique
  neighbor in $H$, so $H, v', v'Pv$ contradicts the minimality of
  $H, x, P$.  So, $v'=v$ and by Lemma~\ref{l:Hv}, $v'$ is adjacent to a neighbor $z$ of $y$ in
  $H$.  If $u=v$, then $\{x, y, z, u\}$ induces a diamond, so $G$ has
  a clique cutset by Lemma~\ref{diamond}.  If $u\neq v$, then by
  Lemma~\ref{l:2holesT}, $G$ has a clique cutset.  Hence, we may
  assume that no node of $P$ is adjacent to $y$.

If $v$ has two adjacent neighbors in $H$, then $x$, $P$ and $H$ form a
pyramid.  So, by Lemma~\ref{l:Hv}, $v$ has a unique neighbor in $H$.
If this neighbor is not adjacent to $y$, then $x$, $P$ and $H$ form a
theta.  Otherwise, $G$ has a clique cutset by Lemma~\ref{l:2holes}.
\end{proof}

\noindent
{\em Proof of Theorem~\ref{th:decWTPFRee}:}
Assume $G$ has no clique cutset. Then by Lemma~\ref{diamond}, $G$ does
not contain a diamond and by Lemma~\ref{l:lgtfc}, we may assume that $G$ contains a claw
$\{v, x, y, z\}$ centered at $v$.  Since $v$ cannot be a cut node,
there exists a path $P$ in $G\setminus v$ whose endnodes are distinct
nodes of $\{x, y, z\}$. We assume that $x, y, z, P$ are chosen
subject to the minimality of $P$.  W.l.o.g. $P$ is a path from $x$ to
$y$, and by the minimality of $P$, it does not go through $z$.

Suppose that no internal node of $P$ is adjacent to $v$. Then
$P\cup \{ v\}$ induces a hole $H$. By Lemma~\ref{l:Hv}, $v$ is the
unique neighbor of $z$ in $H$. But this contradicts
Lemma~\ref{l:HU}. Therefore an internal node of $P$ is adjacent to~$v$.

 Let $v'$ be any internal node of $P$ that is adjacent to $v$.
 We now show that $N(v')\cap \{ x,y,z \} =\{ z\}$. Since $G$ does not contain a diamond,
 w.l.o.g. $v'$ is not adjacent to $x$. If $z$ is not adjacent to $v'$, then $x,v',z$ and $xPv'$
 contradict our choice of $x,y,z$ and $P$. So $z$ is adjacent to $v'$.
 Since $\{ v,y,v',z\}$ does not induce a diamond, it follows that $v'$
 is not adjacent to $y$.  So, as claimed $N(v') \cap \{x, y, z\} =
 \{z\}$. Now, $\{v, x, y, v'\}$ is a claw centered at $v$ and the path
 $xPv'$  contradicts the minimality of $x, y, z$ and $P$. \hfill
$\Box$

\section{Only-pyramid graphs}
\label{sec:onlypyramid}

In this section we prove Theorem~\ref{ptwfree}.  The proof mostly
relies on previously proved theorems and some terminology is needed to
state them.

We say that a clique is {\it big} if it is of size at least
3.  Let $L$ be the line graph of a tree. By Theorem~\ref{th:HararyH} and
Lemma~\ref{l:duk}, every edge
of $L$ belongs to exactly one maximal clique, and every node of $L$
belongs to at most two maximal cliques. The nodes of $L$ that belong
to exactly one maximal clique are called {\em leaf nodes}. In the graph obtained
from $L$ by removing all edges in big cliques, the connected
components are chordless paths (possibly of length 0). Such a path is
an {\em internal segment} if it has its endnodes in distinct big
cliques (when $P$ is of length 0, it is called an internal segment
when the node of $P$ belongs to two big cliques). The other paths $P$
are called {\em leaf segments}. Note that one of the endnodes of a
leaf segment is a leaf node.

A {\em nontrivial basic pyramid graph} $R$ is defined as follows: $R$
contains two adjacent nodes $x$ and $y$, called the {\em special
  nodes}. The graph $L$ induced by $R\setminus \{ x,y\}$ is the line
graph of a tree and contains at least two big cliques.  In $R$, each
leaf node of $L$ is adjacent to exactly one of the two special
nodes, and no other node of $L$ is adjacent to the special
nodes. Furthermore, no two leaf segments of $L$ with leaf nodes
adjacent to the same special node have their other endnode in the same
big clique (this is referred to in the rest of the section as the
\emph{uniqueness condition}). The {\em internal segments} of $R$ are the internal
segments of $L$, and the {\em leaf segments} of $R$ are the leaf
segments of $L$ together with the node in $\{ x,y\}$ to which the leaf
segment is adjacent to.  $R$ is {\em long} if all the leaf segments
are of length greater than 1.

An {\em extended nontrivial basic pyramid graph} is any graph $R^*$
obtained from a nontrivial basic pyramid graph $R$ with special nodes
$x$ and $y$ by adding nodes $u_1, \dots, u_k$ satisfying the
following: for every $i=1, \dots, k$, there exists a big clique $K_i$
of $R$ and some $z_i\in \{ x,y\}$ such that
$N(u_i)\cap V(R)=V(K_i)\cup \{ z_i\}$.  Note that $u_i$ is the
center of a wheel of $R$.

A wheel $(H,x)$ is an {\em even wheel} if $x$ has an even number of
neighbors on~$H$.  A node cutset $S$ of a graph $G$ is a {\em
  bisimplicial cutset} if for some $x\in S$,
$S\subseteq N(x)\cup\{x\}$ and $S\setminus \{ x\}$ is a disjoint union
of two cliques.

\begin{theorem}\label{dehf}
{\em (Kloks, M\"uller, Vu\v{s}kovi\'c \cite{kmv:evenhole})}
A connected (diamond, 4-hole, prism, theta, even wheel)-free graph is either
one of the following graphs:
\begin{itemize}
\item a clique,
\item a hole,
\item a long pyramid, or
\item an extended
nontrivial basic pyramid graph,
\end{itemize}
or it has a bisimplicial cutset or a 2-join.
\end{theorem}

\begin{lemma}\label{4hole}
If $G$ is a connected only-pyramid graph that contains a 4-hole, then either $G$
is a 4-hole or it has a clique cutset.
\end{lemma}
\begin{proof}
  Let $H=x_1x_2x_3x_4x_1$ be a 4-hole of $G$, and assume that
  $G\neq H$.  Let $C$ be a connected component of $G\setminus H$.
  Suppose that two nonadjacent nodes of $H$, say $x_1$ and $x_3$, both
  have a neighbor in $C$. Let $P$ be a path in $C$ such that $x_1Px_3$
  is a chordless path. W.l.o.g. we may assume that $P$ is minimal such
  path.  Suppose that both $x_2$ and $x_4$ have a neighbor in
  $P$. Then, by the choice of $P$ and since $G$ is wheel-free, $P$ is
  of length at least 1, $x_2$ and $x_4$ are adjacent to different
  endnodes of $P$, and they each have a unique neighbor in $P$. But
  then $P\cup H$ induces a prism.  So w.l.o.g. $x_2$ does not have a
  neighbor in $P$. But then $P\cup H$ induces a theta or a wheel.
  Therefore, for some edge $uv$ of $H$, $N(C)\cap H=\{ u,v\}$, and so
  $G$ has a clique cutset.
\end{proof}

\medskip

\noindent
{\em Proof of Theorem~\ref{ptwfree}:}
Let $G$ be an only-pyramid graph that does not have a clique cutset
and is not a hole.  By Lemmas~\ref{diamond} and~\ref{4hole}, $G$ is
(diamond, 4-hole)-free.  Since a bisimplicial cutset is a star cutset,
by Theorem~\ref{dehf} and Lemma \ref{star}, it is enough to show that
if $G$ is an extended nontrivial basic pyramid graph, then it is
pyramid-basic.

As noted above, if a wheel-free graph $G$ is
an extended nontrivial basic pyramid graph, then it is a long
nontrivial basic pyramid graph.  So, $G$ is obtained from the
line graph of a tree $T$ by adding two nodes $x$ and $y$ as
explained above.

Let us check that $T$ is safe.  First, note that by the uniqueness
condition in the definition of nontrivial basic pyramid graphs, it
cannot be that more than two leaf segments have the non-leaf end in
the same big clique.  This means that in $T$,
there does not exists three pendant edges that are siblings, every
pendant edge of $T$ has at most one sibling.  To check that $T$ is
safe, it remains to check that for every node $u$ of degree 1, the
neighbor $v$ of $u$ has degree at most 2.  So, suppose for a
contradiction that $T$ contains a node $u$ of degree 1 whose
neighbor $v$ has degree at least~3.  So, the edge $uv$ is a node $c$
of $L(T)$. Node $c$ is a leaf node of $L(T)$, so it must be adjacent
to $x$ or $y$, say to $x$.  Also, $c$ is adjacent to two nodes $a$,
$b$ of some big clique of $G$.  Since edge $c$ of $T$ has at most one
sibling, we may assume up to symmetry that $a$ is the end of an
internal segment of $L(T)$.  So, there are two node-disjoint paths
$P= a\dots x$ and $Q=y\dots b$ in $L(T)$: $P$ starts by the internal
segment ending at $a$, reaches another big clique, and then any leaf
segment in that part of the tree with an end $x$, while $Q$ starts from the segment
ending at $b$ (if it is a leaf segment, it is linked to $y$ by the
uniqueness condition, otherwise, it can be linked to $x$ or $y$, and
we choose $y$). The union of $P$ and $Q$ forms a hole, and $c$ has
three neighbors in that hole, namely $a$, $b$ and $x$.  This proves
that $T$ is safe.

Now, the uniqueness condition shows that $G$ is in fact a pyramid-basic
graph.  \hfill $\Box$

\section{2-joins}
\label{sec:2joins}

In this section, we describe more closely the structure of the 2-joins
and the almost 2-joins that actually occur in our classes of graphs.
An almost 2-join with a split $(X_1, X_2, A_1, A_2, B_1, B_2)$ in a
graph $G$ is \emph{consistent} if the following statements hold for
$i=1, 2$:

\begin{enumerate}
\item\label{i1} Every component of $G[X_i]$ meets both $A_i$, $B_i$.
\item\label{i3} Every node of $A_i$ has a non-neighbor in $B_i$.
\item\label{i4} Every node of $B_i$ has a non-neighbor in $A_i$.
\item\label{i6} Either both $A_1$, $A_2$ are cliques, or one of $A_1$ or $A_2$ is
  a single node, and the other one is a disjoint union of cliques.
\item\label{i7} Either both $B_1$, $B_2$ are cliques, or one of $B_1$, $B_2$ is
  a single node, and the other one is a disjoint union of cliques.
\item\label{i8} $G[X_i]$ is connected.
\item\label{i10} For every node $v$  in $X_i$, there exists a path in $G[X_i]$
  from $v$ to some node of $B_i$ with no internal node in $A_i$.
\item\label{i9} For every node $v$  in $X_i$, there exists a path in $G[X_i]$
  from $v$ to some node of $A_i$ with no internal node in $B_i$.
\end{enumerate}

Note that the definition contains redundant statements (for instance,~\ref{i8}
implies~\ref{i1}), but it is convenient to list properties separately as above.

\begin{lemma}
  \label{l:consistent}
  If $G$ is a (theta, wheel)-free graph with no clique cutset, then
  every almost 2-join of $G$ is consistent.
\end{lemma}

\begin{proof}
By Lemma~\ref{diamond}, $G$ contains no diamond, and by
Lemma~\ref{star}, it has no star cutset. This is going to be used
repeatedly in the proofs below.  Let $(X_1,X_2,A_1,A_2,B_1,B_2)$ be a
split of an almost 2-join of $G$.

 To prove~\ref{i1}, suppose for a contradiction that some
  connected component $C$ of $G[X_1]$ does not intersect $B_1$ (the
  other cases are symmetric).  If there is a node $c\in C\sm A_1$
  then for any node $u\in A_2$, we have that $\{u\}\cup A_1$ is a
  star cutset that separates $c$ from $B_1$.  So, $C\subseteq A_1$.
  If $|A_1| \geq 2$ then pick any node $c\in C$ and a node $c'\neq
  c$ in $A_1$.  Then $\{c'\} \cup A_2$ is a star cutset that separates
  $c$ from $B_1$.  So, $C= A_1 = \{c\}$.  Hence, there exists some
  component of $G[X_1]$ that does not intersect $A_1$, so by the same
  argument as above we deduce $|B_1| = 1$ and the unique node of
  $B_1$ has no neighbor in $X_1$.  Since $|X_1|\geq 3$, there is a
  node $u$ in $C_1$.  For any node $v$ in $X_2$, $\{v\}$ is a star
  cutset of $G$ that separates $u$ from $A_1$, a contradiction.

 To prove~\ref{i3} and~\ref{i4}, consider a node
  $a\in A_1$ complete to $B_1$ (the other cases are symmetric).  If
  $A_1\cup C_1 \neq \{a\}$ then $B_1\cup A_2 \cup \{a\}$ is a star
  cutset that separates $(A_1\cup C_1) \sm \{a\}$ from $B_2$, a
  contradiction.  So, $A_1\cup C_1 = \{a\}$ and $|B_1| \geq 2$ because
  $|X_1|\geq 3$.  Let $b\neq b' \in B_1$.  So, $\{b, a\} \cup B_2$ is a
  star cutset that separates $b'$ from $A_2$, a contradiction.

  We now prove~\ref{i6}. If $|A_i| = 1$ then $A_{3-i}$ contains no
  path of length~2 since $G$ contains no diamond. It follows that $A_{3-i}$ is
  a disjoint union of cliques. We may therefore assume that
  $|A_1|, |A_2| \geq 2$.  If $A_i$ is not a clique, then it contains
  two non-adjacent nodes that form a diamond together with any edge of
  $A_{3-i}$.  It follows that $A_{3-i}$ is a stable set, and by
  symmetry, so is $A_i$.  Since $K_{2, 3}$ is
  a theta, we have $|A_1|= |A_2| = 2$.

  Let $A_1 = \{a_1, a'_1\}$ and $A_2= \{a_2, a'_2\}$.  Suppose that
  $a_1$ and $a'_1$ are in the same connected component of $G[X_1]$.
  Then, a path of $G[X_1]$ from $a_1$ to $a'_1$ together with $a_2$
  and $a'_2$ form a theta, a contradiction.  It follows that $a_1$ and
  $a'_1$ are in different connected components of $G[X_1]$.
  By~\ref{i1}, it follows that $G[X_1]$ has precisely two connected
  components.  By the same argument, $G[X_2]$ also has precisely two
  connected components.  It follows that $|B_1|, |B_2|\geq 2$, and by
  the same proof as in the paragraph above, $B_1$ and $B_2$ are stable
  sets of size 2.  By~\ref{i1}, there is a chordless path $P_1$ in
  $G[X_1]$ from $a_1$ to some node of $B_1$, that we denote by $b_1$.
  There are similar paths $P'_1 = a'_1 \dots b'_1$,
  $P_2 = a_2 \dots b_2$ and $P'_2 = a'_2 \dots b'_2$.  If $P_1$ has
  length at least~2 (meaning that $a_1$ and $b_1$ are non-adjacent),
  then $\{ a_1,a_1',b_1\} \cup V(P_2)\cup V(P_2')$ contains a
  $3PC(a_2,a_2')$. Therefore $P_1$ has length 1, and by symmetry so do
  $P_1',P_2$ and $P_2'$.  But then
  $\{ a_1, a_1', a_2, a_2', b_1, b_1', b_2' \}$ induces a wheel with
  center $a_2'$, a contradiction.  This completes the proof
  of~\ref{i6} and the proof of~\ref{i7} is similar.

  To prove~\ref{i8} suppose by contradiction and up to symmetry
  that $G[X_1]$ is disconnected.  By~\ref{i1}, $G[A_1]$ and $G[B_1]$
  must be disconnected, so by~\ref{i6} and~\ref{i7}, they are disjoint
  union of cliques and $A_2$ and $B_2$ are both made of a single node,
  say $a_2$ and $b_2$ respectively.  By~\ref{i1} there exists a
  chordless path $P$ in $G[X_2]$ from $a_2$ to $b_2$. By~\ref{i3} this
  path is of length at least~2.  Therefore, by considering three paths
  from $a_2$ to $b_2$ (one that goes through a component of $X_1$, one
  that goes through another component of $X_1$, and $P$), we obtain a
  theta, a contradiction.

  Suppose~\ref{i10} does not hold.  So, up to symmetry there exists a
  node $v\in X_1$ such that every path in $G[X_1]$ from $v$ to $B_1$
  has an internal node in $A_1$.  Note in particular that
  $v\notin B_1$.  Also, $A_1 = \{v\}$ is impossible, because if so,
  by~\ref{i8} there exists a path in $G[X_1]$ from $v$ to $B_1$, and
  since $A_1 = \{v\}$, this path has no internal node in $A_1$, a
  contradiction.  It follows that $A_2 \cup A_1 \sm \{v\}$ is a cutset
  that separates $v$ from $B_1$, and since $A_1 \neq \{v\}$, this
  cutset contains at least one node of $A_1$.  If $A_1$ is a clique,
  then $A_2 \cup A_1 \sm \{v\}$ is a star cutset (centered at any
  node of $A_1\sm \{v\}$) that separates $v$ from the rest of the
  graph, a contradiction.  Since $A_1$ is not a clique, by~\ref{i6} it
  is a disjoint union of cliques and $A_2$ is single node $a$.  It
  follows that $\{a\} \cup A_1 \sm \{v\}$ is a star cutset centered at
  $a$, a contradiction.  Hence~\ref{i10} holds, and by an analogous proof,
  so does~\ref{i9}.
\end{proof}

We now define the blocks of decomposition of a graph with respect  to a
2-join.  Let $G$ be a graph and $(X_1, X_2)$ a 2-join of $G$.  The
\emph{blocks of decomposition} of $G$ with respect to $(X_1, X_2)$ are
the two graphs $G_1$ and $G_2$ that we describe now.  We obtain $G_1$
from $G$ by replacing $X_2$ by a \emph{marker path} $P_2= a_2 c_2
b_2$, where $a_2$ is a node complete to $A_1$, $b_2$ is a node complete
to $B_1$, and  $c_2$ has no neighbor in $X_1$.  The
block $G_2$ is obtained similarly by replacing $X_1$ by a marker path
$P_1 = a_1c_1b_1$ of length $2$.

\begin{lemma}
  \label{keepCons}
  Let $(X_1,X_2)$ be a consistent 2-join of
  a graph $G$, and let $G_1$ and $G_2$ be the blocks of decomposition
  of $G$ with respect to $(X_1, X_2)$.  Then, for $i=1, 2$,
  $(X_i, V(P_{3-i}))$ is a consistent almost 2-join of $G_i$.
\end{lemma}

\begin{proof}
  Obviously, $(X_i, V(P_{3-i}))$ is an almost 2-join of $G_i$ (but not
  a 2-join, the side $V(P_{3-i})$ violates the additional condition in
  the definition of 2-joins).  It is consistent, because all the
  conditions to be checked in $X_i$ are inherited from the fact that
  they hold in $G$, and the conditions in $V(P_{3-i})$ are trivially
  true.
\end{proof}

\begin{lemma}
  \label{l:keepKfree}
  Let $G$ be a graph with a consistent 2-join $(X_1, X_2)$ and $G_1$,
  $G_2$ be the blocks of decomposition with respect to this 2-join.
  Then, $G$ has no clique cutset if and only if $G_1$ and $G_2$ have
  no clique cutset.
\end{lemma}

\begin{proof}
  We prove an equivalent statement: $G$ has a clique cutset if and
  only if $G_1$ or $G_2$ has a clique cutset.

  Suppose first that $G$ has a clique cutset $K$.  By the definition
  of a 2-join and up to symmetry, either $K \subseteq X_1$, or
  $K\subseteq A_1 \cup A_2$.  In the first case, by condition~\ref{i8}
  in the definition of consistent 2-joins, $G[X_2]$ is connected, so
  $X_2$ is included in some component of $G\sm K$.  It follows that
  $K$ is a clique cutset of $G_1$.  In the second case, by
  condition~\ref{i10} of consistent 2-joins, every node of $G\sm K$
  can be linked to a node of $B_1 \cup B_2$ by a path that avoids
  $K$.  So, $K$ is not a cutset, a contradiction.

  Conversely, suppose up to symmetry that $G_1$ has a clique cutset
  $K$.  By Lemma~\ref{keepCons}, $(X_1, V(P_{2}))$ is a consistent
  almost 2-join of $G_1$.  So, by exactly the same proof as in the
  paragraph above, we can prove that $G$ has a clique cutset.
\end{proof}

We now need to study how a hole may overlap a consistent almost 2-join
of a graph.  So, let $G$ be graph, $(X_1, X_2, A_1, A_2, B_1, B_2)$ a
split of a consistent almost 2-join of $G$, and $H$ a hole of $G$.
Because of the adjacencies in an almost 2-join, $H$ must be one of the
following five types:

\begin{description}
  \item[Type 0]: for some $i\in \{1, 2\}$, $V(H) \subseteq X_i$.
  \item[Type 1A]: for some $i\in \{1, 2\}$, $H= ap_1 \dots p_ka$ where
    $k\geq 3$, $p_2, \dots, p_{k-1} \in X_i\sm A_i$, $a\in A_{3-i}$,
    and $\{p_1, p_k\} \subseteq A_i$.
  \item[Type 1B]: for some $i\in \{1, 2\}$, $H= bp_1 \dots p_kb$ where
    $k\geq 3$, $p_2, \dots, p_{k-1} \in X_i\sm B_i$, $b\in B_{3-i}$,
    and $\{p_1, p_k\} \subseteq B_i$.
  \item[Type 2]: for some $i\in \{1, 2\}$,
    $H= ap_1 \dots p_kbq_1 \dots q_la$ where $k\geq 2$, $l\geq 2$,
    $p_2, \dots, p_{k-1},  q_2, \dots, q_{l-1} \in X_i\sm (A_i \cup B_i)$, $a\in A_{3-i}$,
    $b\in B_{3-i}$, $\{p_1, q_l\} \subseteq A_i$, and
    $\{p_k, q_1\} \subseteq B_i$.
    \item[Type 3]: $H= p_1 \dots p_k q_1 \dots q_l p_1$ where $k, l
      \geq 2$, $p_2, \dots, p_{k-1} \in X_1\sm (A_1 \cup B_1)$, $q_2,
      \dots, q_{l-1}  \in X_2\sm (A_2 \cup B_2)$, $p_1\in A_1$, $p_k\in
      B_1$, $q_1 \in B_2$, $q_l\in A_2$.
\end{description}

Note that if $(X_1, X_2)$ is a 2-join (rather that just an almost
2-join) and $H$ a hole  of type 0, 1A, 1B or 2, then up to the
replacement of $a$ and/or $b$ by a marker node, $H$ is a hole of
$G_1$ or a hole of $G_2$, and we simply denote this hole by $H$ (with a
slight abuse, due to the replacement of a node by a marker node).
If $H$ is of type~3, then by replacing $q_1\dots q_l$ by the marker
path $P_2$, we obtain a hole $H_1$ of $G_1$, and by replacing
$p_1\dots p_l$ by the marker path $P_1$, we obtain a hole $H_2$ of
$G_2$. We will use this notation in what follows.

Let $G$ be a graph that contains a consistent 2-join $(X_1,X_2)$, and let
$G_1$ and $G_2$ be the corresponding blocks of decomposition.
Consider $G_1$ (analogous statements hold for $G_2$).
$(X_1,V(P_2))$ is not a 2-join of $G_1$
 (but is still an almost 2-join, and a
consistent one by Lemma~\ref{keepCons}). Suppose $G_1$ contains a hole
$H_1$.
 Then, as above, from $H_1$ we may build a hole
$H$ of $G$.  If $H_1$ is of type 0 or 1A or 1B, this is straightforward.
If $H_1$ is of type~2, then we need to be careful when replacing $a$
and $b$ by nodes from $X_2$: the new nodes need to be
non-adjacent, but the existence of such nodes is guaranteed by the
condition~\ref{i3} in the definition of consistent 2-joins.  If $H_1$
is of type~3, then it contains the marker path $P_2$, but a hole $H$
in $G$ can be obtained by replacing this marker path by a shortest
path linking the special sets of $X_2$ (whose existence follows from
the condition~\ref{i1}).

In the proofs of the next lemmas, we will use repeatedly the notation
and constructions from the two paragraphs above.

\begin{lemma}
  \label{l:keepPrism}
  Let $G$ be a graph with a consistent
  2-join $(X_1, X_2)$. Let $G_1$ and $G_2$ be the blocks of decomposition
  of $G$ with respect to $(X_1, X_2)$.  Then, $G$ is
  prism-free if and only if $G_1$ and $G_2$ are both prism-free.
\end{lemma}

\begin{proof}
  We prove the equivalent statement  ``$G$ contains a prism if and only if $G_1$ or
  $G_2$ contains a prism''.

  Suppose that $G$ contains a prism $T$.  Note that a prism contains
  three holes, and we denote by $H$ a hole of $T$ whose type is
  maximal (we order types as follows: $0 < 1A, 1B < 2 < 3$).  Hence,
  $T$ is made of $H$, together with a path $P = u\dots v$ of length at
  least~1, node disjoint from $H$, and $u$ and $v$ are adjacent to two
  disjoint edges of $H$.  Note that $P$ is contained in the two holes
  of $T$ that are different from $H$.  If $H$ is of type~0, then up to
  symmetry $V(H)\subseteq X_1$, and by the maximality of $H$, all
  holes of $T$ are of type~0 and $V(T) \subseteq X_1$.  So, $T$ is
  also a prism of $G_1$. If $H$ is of type 1A, then up to symmetry
  $H= ap_1 \dots p_ka$, $p_2, \dots, p_{k-1} \in X_1\sm A_1$ and
  $a\in A_2$.  Hence, $A_1$ contains non-adjacent nodes, so by the
  condition~\ref{i6} in the definition of consistent 2-joins,
  $A_2=\{a\}$.  It follows by the maximality of $H$ that $P$ does not
  contain a subpath from $A_2$ to $B_2$, so
  $V(P) \subseteq X_1\cup B_2$, and if $P$ overlaps $B_2$, then the
  condition~\ref{i7} implies that $|B_2| = 1$.  Hence, after possibly
  replacing the node of $B_2$ by a marker node, $H$ and $P$ form a
  prism of $G_1$.  The case when $H$ is of type 1B is symmetric to the
  previous one.  The proof is similar when $H$ is of type~2.  If $H$
  is of type 3, then we claim that $P$ is in $X_1$ or in
  $X_2$. Otherwise, $P$ must contains adjacent nodes in $X_1$ and
  $X_2$, up to symmetry in $A_1$ and $A_2$ respectively.  Hence,
  $|A_1|, |A_2| \geq 2$, so by condition~\ref{i6}, $A_1$ and $A_2$
  are both cliques, a contradiction since a prism
  contains no $K_4$.  So, up to symmetry $P$ is in $X_1$, and $H_1$
  and $P$ form a prism of $G_1$.

  The proof of the converse statement is analogous: we start
  with a prism of $G_1$ or $G_2$, and according to the type of a maximal
  hole of the prism, we build a prism of $G$.
\end{proof}

\begin{lemma}
  \label{l:keepTW}
  Let $G$ be a graph with a consistent 2-join $(X_1, X_2)$. Let $G_1$
  and $G_2$ be the blocks of decomposition of $G$ with respect to
  $(X_1, X_2)$.  Then, $G$ is (theta, wheel)-free if and only if $G_1$
  and $G_2$ are both (theta, wheel)-free.
\end{lemma}

\begin{proof}
  We first prove that $G$ contains a wheel if and only if $G_1$ or
  $G_2$ contains a wheel.

  Suppose that $G$ contains a wheel with rim $H$ and center $v$.  If
  $H$ is of type~0, then up to symmetry $V(H)\subseteq X_1$.  Then,
  $H$ is also the rim of a wheel of $G_1$ (the center is $v$, or
  possibly a marker node).  If $H$ is of type 1A, then up to
  symmetry $H= ap_1 \dots p_ka$, $p_2, \dots, p_{k-1} \in X_1\sm A_1$
  and $a\in A_2$.  Hence, $A_1$ contains non-adjacent nodes, so by
  the condition~\ref{i6} in the definition of consistent 2-joins,
  $A_2=\{a\}$.  It follows that $v\in X_1$, and that $(H, v)$ is a
  wheel of $G_1$.  The case when $H$ is of type 1B is symmetric to the
  previous one.  If $H$ is of type 2, then up to symmetry
  $H= ap_1 \dots p_kbq_1 \dots q_la$
  $p_2, \dots, p_{k-1}, q_2, \dots, q_{l-1} \in X_1\sm (A_1 \cup
  B_1)$,
  $a\in A_{2}$, $b\in B_{2}$, $\{p_1, q_l\} \subseteq A_1$, and
  $\{p_k, q_1\} \subseteq B_1$.  By the conditions~\ref{i6}
  and~\ref{i7} in the definition of consistent 2-joins, $A_2=\{a\}$
  and $B_2=\{b\}$.  It follows that $v\in X_1$, and that $(H, v)$ is a
  wheel of $G_1$.  If $H$ is of type 3, then up to symmetry we suppose
  that $v\in X_1$.  We observe that $(H_1, v)$ is a wheel of $G_1$.

  The proof of the converse statement is analogous: we start
  with a wheel of $G_1$ or $G_2$, and according to the type of the rim, we
  build a wheel of $G$.

 We now prove that $G$ contains a theta if and only if $G_1$ or
  $G_2$ contains a theta.

  Suppose that $G$ contains a theta $T$.  Note that a theta contains
  three holes, and we denote by $H$ a hole of $T$ whose type is
  maximal w.r.t.\ the order defined in the proof of
  Lemma~\ref{l:keepPrism}. Hence, $T$ is made of $H$, together with a
  path $P = u\dots v$ of length at least~2, where $u$ and $v$ are non
  adjacent nodes of $H$.  If $H$ is of type~0, then up to symmetry
  $V(H)\subseteq X_1$, and by the maximality of $H$, all holes of $T$
  are of type~0 and $V(T) \subseteq X_1$.  So, $T$ is also a theta of
  $G_1$. If $H$ is of type 1A, then up to symmetry
  $H= ap_1 \dots p_ka$, $p_2, \dots, p_{k-1} \in X_1\sm A_1$ and
  $a\in A_2$.  Hence, $A_1$ contains non-adjacent nodes, so by the
  condition~\ref{i6} in the definition of consistent 2-joins,
  $A_2=\{a\}$.  It follows by the maximality of $H$ that $P$ does not
  contain a subpath from $A_2$ to $B_2$, so
  $V(P) \subseteq X_1\cup B_2$, and if $P$ overlap $B_2$, then the
  condition~\ref{i7} implies that $|B_2| = 1$.  Hence, after possibly
  replacing the node of $B_2$ by a marker node, $H$ and $P$ form a
  theta of $G_1$.  The case when $H$ is of type 1B is symmetric to the
  previous one.  The proof is similar when $H$ is of type~2.  If $H$
  is of type 3, then we claim that the interior of $P$ is in $X_1$ or
  in $X_2$. Otherwise, the interior of $P$ must contains adjacent
  nodes in $X_1$ and $X_2$, up to symmetry in $A_1$ and $A_2$
  respectively.  This violates the condition~\ref{i6}, a contradiction
  that proves our claim.  So, up to symmetry the interior of $P$ is in
  $X_1$, and $H_1$ and the interior of $P$ form a theta of $G_1$.

  The proof of the converse statement is analogous: we start with a
  theta of $G_1$ or $G_2$, and according to the type of a maximal hole
  of the theta, we build a theta of $G$.
\end{proof}

\begin{lemma}
  \label{l:atl4}
  If a graph $G$ has a consistent 2-join $(X_1, X_2)$, then
  $|X_1|, |X_2| \geq 4$.
\end{lemma}

\begin{proof}
  Suppose for a contradiction that $|X_1| = 3$.  Up to symmetry we
  assume $|A_1| = 1$, and let $a_1$ be the unique node in $A_1$.  By
  the condition~\ref{i4} in the definition of consistent 2-joins, every
  node of $B_1$ has a non-neighbor in $A_1$.  Since $A_1 = \{a_1\}$,
  this means that $a_1$ has no neighbor in $B_1$.  By~\ref{i1},
  $G[X_1]$ is a path of length 2 whose interior is in $C_1$.  This
  contradicts the definition of a 2-join (note that this does not
  contradict the definition of an almost 2-join).
\end{proof}

\section{Algorithms}
\label{sec:algo}

We are now ready to describe our recognition algorithms based on
decomposition by
clique cutsets and 2-joins.  When a graph $G$ has a clique cutset $K$,
its node set can be partitioned into nonempty sets $A$, $K$, and $B$ in
such a way that there are no edges between $A$ and $B$.  We call such
a triple a \emph{split} for the clique cutset.  When $(A, K, B)$ is a
split for a clique cutset of a graph $G$, the blocks of decomposition
of $G$ with respect to $(A, K, B)$ are the graphs $G_A=
G[A\cup K]$ and $G_B= G[K \cup B]$.

\begin{lemma}
  \label{l:kpc}
  Let $G$ be a graph and $(A, K, B)$ be a split for a clique cutset of
  $G$.  Then, $G$ contains a prism (resp.\ a pyramid, a theta, a
  wheel) if and only if one of the blocks of decomposition $G_A$ or
  $G_B$ contains a prism (resp.\ a pyramid, a theta, a wheel).
\end{lemma}

\begin{proof}
  Follows directly from the fact that a Truemper configuration has no
  clique cutset.
\end{proof}

A \emph{clique cutset decomposition tree} for a graph $G$ is a rooted
tree $T$ defined as
follows.
\begin{itemize}
\item The root of $T$ is $G$.
 \item Every non-leaf node of $T$ is a graph $G'$ that contains a clique
    cutset $K$ with split $(A, K, B)$ and  the children of $G'$ in $T$ are the blocks
    of decomposition $G'_A$ and $G'_B$ of $G$ with respect to $(A, K, B)$.
  \item Every leaf of $T$ is a graph with no clique cutset.
  \item $T$ has at most $n$ leaves.
  \end{itemize}

\begin{theorem}[Tarjan~\cite{tarjan:clique}]
  \label{th:tarjan}
  A clique cutset decomposition tree of an input graph $G$ can be
  computed in time $O(nm)$.
\end{theorem}

A \emph{consistent 2-join decomposition tree} for a graph $G$ is a
rooted tree $T$ defined as follows.
\begin{itemize}
\item The root of $T$ is $G$.
\item Every non-leaf node of $T$ is a graph $G'$ that contains a
  consistent 2-join with split $(X_1, X_2, A_1, A_2, B_1, B_2)$ and the children
  of $G'$ in $T$ are the blocks of decomposition  $G'_1$ and $G'_2$
  with respect to $(X_1, X_2, A_1, A_2, B_1, B_2)$.
  \item Every leaf of $T$ is a graph with no 2-join, or a graph with a
    non-consistent 2-join (and is identified as such).
  \item $T$ has at most $O(n)$ nodes.
  \end{itemize}

\begin{theorem}
  \label{th:build2joinT}
  A consistent 2-join decomposition tree of an input graph $G$ can be
  computed in time $O(n^3m)$.
\end{theorem}

\begin{proof}
  Here is an algorithm that outputs a tree $T$.  We run an algorithm
  from~\cite{ChHaTrVu:2-join} that outputs in time~$O(n^2m)$ a split
  of a 2-join of $G$, or certifies that no 2-joins exists (warning:
  what we call here a 2-join is called in~\cite{ChHaTrVu:2-join} a
  \emph{non-path 2-join}).  If $G$ has no 2-join, then $G$ is declared
  to be a leaf of $T$.  If a split $(X_1, X_2, A_1, A_2, B_1, B_2)$ is
  outputted, we check whether $(X_1, X_2)$ is consistent (this can be
  easily done in time $O(nm)$, all the conditions in the definition
  of consistent 2-joins are easy to check).  If the 2-join is not
  consistent, then $G$ is declared to be a leaf of $T$.  Otherwise, we
  compute the blocks of decomposition $G_1$ and $G_2$ of $G$ with
  respect to $(X_1, X_2, A_1, A_2, B_1, B_2)$, and run the algorithm
  recursively for $G_1$ and $G_2$.

  The algorithm is clearly correct.  Here is the complexity
  analysis. We may assume that the input graph has at least 7 nodes
  (otherwise, we look directly for the tree in constant time).  Note
  that by Lemma~\ref{l:atl4}, at every recursive call, the size of the
  graph decreases, so the algorithm terminates.  Also, every graph
  involved in the algorithm has at least seven nodes.  We denote by
  $f(G)$ the number of calls to the algorithm for a graph $G$ on $n$
  nodes.  We show by induction that $f(G) \leq 2n-13$.  If $G$ is a
  leaf of $T$, this is true because $f(G) = 1$, and since $n\geq 7$,
  we have $2n-13\geq 1$.  If $G$ is not a leaf, then it has a 2-join
  $(X_1, X_2)$ and we set $n_1 = |X_1|$ and $n_2 = |X_2|$.  Note that
  $n=n_1+n_2$ and that the blocks of decomposition $G_1$ and $G_2$
  have respectively $n_1+3$ and $n_2+3$ nodes. Since there is one call
  to the algorithm plus at most $f(G_1) + f(G_2)$ recursive calls, by
  the induction hypothesis we have:

  $$
  f(G) \leq f(G_1) + f(G_2) + 1
  \leq 2(n_1 +3) -13 + 2(n_2 +3) -13 + 1
  = 2n -13.
  $$

  So, there are at most $2n-13$ calls to an algorithm of complexity
  $O(n^2m)$.  The overall complexity is therefore $O(n^3m)$.  Since
  the number of nodes of the tree is bounded by the number of
  recursive calls, $T$ has at most $O(n)$ nodes.
\end{proof}

We need to recognize in polynomial time the basic classes of our
theorems.

\begin{lemma}
  \label{l:testOp}
  There is an $O(n^2m)$-time algorithm that decides whether an input graph is
  the line graph of a triangle-free chordless graph (resp.\ a
  pyramid-basic graph, a long pyramid, a clique, a hole).
\end{lemma}

\begin{proof}
  Note first that deciding whether a graph $G$ is a line graph, and if
  so computing a graph $R$ such that $G=L(R)$ can be performed in time
  $O(n+m)$ as shown in~\cite{lehot:root,roussopoulos:linegraphe}.
  Deciding whether a graph is chordless can be done easily in time
  $O(nm + m^2)$: for every edge $uv$, compute in time $O(n+m)$  the blocks
  of $G\sm uv$ by the classical algorithm
  from~\cite{tarjan:dfs}.  Then, check whether $u$ and $v$ are in the
  same 2-connected block (this holds if and only if $uv$ is a chord of
  some cycle of $G$).  Deciding whether a graph is triangle-free
  can be performed trivially in time $O(nm)$.  By combining all this,
  we test in time $O(n^2m)$ whether a graph is a line graph of a
  triangle-free chordless graph.

  To test whether a graph is a pyramid-basic graph, for every edge  $xy$,
  we test whether $G\sm \{x, y\}$ is the line-graph of a tree, and if
  so, we compute the tree, and check whether it is safe. Checking whether
  $x$ and $y$ satisfy the requirement of the definition of pyramid-basic
  graphs is then easy.

  Checking whether a graph is a long pyramid, a hole or a clique is
  trivial.
\end{proof}

\begin{theorem}
  \label{th:recoOprism}
  There exists a $O(n^3m)$ time algorithm that decides whether an
  input graph $G$ is only-prism.
\end{theorem}

\begin{proof}
  We run the algorithm of Theorem~\ref{th:tarjan}.  This gives a list
  of $O(n)$ graphs (the leaves of the decomposition tree) that have no
  clique cutsets, and by Lemma~\ref{l:kpc}, $G$ is only-prism if and
  only if so are all graphs of the list.  By the algorithm from
  Lemma~\ref{l:testOp}, we test whether all graphs from the list are
  line graphs of triangle-free chordless graphs.  If so, $G$ is
  only-prism by Lemma~\ref{l:kpc}, and the algorithm outputs ``$G$ is
  only-prism''.  If one graph from the list fails to be the line graph
  of a triangle-free chordless graph, then since it has no clique
  cutset, it is not only-prism by Theorem~\ref{th:decWTPFRee}.  So,
  the algorithms outputs ``$G$ is not only-prism''.  In the
  worst case, we run $O(n)$ times an algorithm of complexity
  $O(n^2m)$.
\end{proof}

\begin{theorem}
    \label{th:recoOpy}
  There exists an  $O(n^4m)$-time algorithm that decides whether an
  input graph $G$ is only-pyramid.
\end{theorem}

\begin{proof}
  We first run the algorithm of Theorem~\ref{th:tarjan}.  This gives a
  list of $O(n)$ graphs (the leaves of the decomposition tree) that
  have no clique cutsets, and by Lemma~\ref{l:kpc}, $G$ is
  only-pyramid if and only if so are all graphs of the list.
  Therefore, it is enough to provide an algorithm for graphs with no
  clique cutsets.

  So, suppose $G$ has no clique cutset.  By
  Theorem~\ref{th:build2joinT}, we build a consistent 2-join
  decomposition tree $T$ of $G$.  By Lemma~\ref{l:keepKfree}, all
  nodes of $T$ are graphs that have no clique cutset.  If one leaf $T$
  has a non-consistent 2-join, then it cannot be an only-pyramid graph
  by Lemma~\ref{l:consistent}.  The algorithm therefore outputs ``$G$
  is not only-pyramid'', the correct answer by
  Lemmas~\ref{l:keepPrism} and~\ref{l:keepTW}.  Now, we may assume
  that all leaves of $T$ have no 2-join.  By Lemma~\ref{l:testOp}, we
  check whether some leaf of $T$ is a long pyramid, a clique, a hole
  or a pyramid-basic graph.  If one leaf fails to be such a graph, then,
  since it has no clique cutset and no 2-join, it cannot be
  only-pyramid by Theorem~\ref{ptwfree}, so again the algorithm
  outputs ``$G$ is not only-pyramid''.  Now, every leaf
  of $T$ can be assumed to be a long pyramid, a clique, a hole or a
  pyramid-basic graph, and it is therefore only-pyramid by
  Lemma~\ref{l:opinC}.  By Lemmas~\ref{l:keepPrism}
  and~\ref{l:keepTW}, $G$ itself is only-pyramid.
\\
\\
  {\bf Complexity analysis.}  The algorithm when there is no
  clique cutset runs in time $O(n^3m)$ because in the worst case, the
  search for a 2-join and the recognition of basic graphs has to be
  done $O(n)$ times.  This algorithm is performed $n$ times in the
  worst case. So, the overall complexity is $O(n^4m)$.
\end{proof}

We now explain how our decomposition theorems can be turned into
structure theorems.

Let $G_1$ be a graph that contains a clique $K$ and $G_2$ a graph
that contains the same clique $K$, and is node disjoint from $G_1$
apart from the nodes of $K$.  The graph $G_1 \cup G_2$ is the graph
obtained from $G_1$ and $G_2$ by \emph{gluing along a clique}.

Let $G_1$ be a graph that contains a path  $a_2 c_2 b_2$ such that
$c_2$ has degree 2, and such that $(V(G_1)\sm \{a_2, c_2, b_2\},
\{a_2, c_2, b_2\})$ is a consistent  almost 2-join of $G_1$.   Let
$G_2, a_1, c_1, b_1$ be defined similarly.  Let $G$ be the graph built
on $(V(G_1)\sm \{a_2, c_2, b_2\})\cup (V(G_2)\sm \{a_1, c_1, b_1\})$
by keeping all edges inherited from $G_1$ and $G_2$, and by adding all
edges between $N_{G_1}(a_2)$ and $N_{G_2}(a_1)$, and all
edges between $N_{G_1}(b_2)$ and $N_{G_2}(b_1)$.  Graph $G$ is said to be
\emph{obtained from $G_1$ and $G_2$ by consistent 2-join
  composition}.   Observe that $(V(G_1)\sm \{a_2, c_2, b_2\},
V(G_2)\sm \{a_1, c_1, b_1\})$ is a 2-join of $G$ and that $G_1$ and
$G_1$ are the blocks of decomposition of $G$ with respect to this
2-join.

With a proof similar to the proof of Theorem~\ref{th:recoOprism}, it
is straightforward to check the following structure theorem. Every
only-prism graph can be constructed as follows:
\begin{itemize}
  \item Start with line graphs of triangle-free chordless graphs.
  \item Glue along a clique previously constructed graphs.
\end{itemize}

Similarly, it can be checked that every only-pyramid graph can be
constructed as follows:
\begin{itemize}
\item Start with long pyramids, holes, cliques and pyramid-basic graphs.
\item Repeatedly use consistent 2-join compositions from previously
  constructed graphs.
\item Glue along a clique previously constructed graphs.
\end{itemize}

\end{document}